\documentclass{amsart} 
\usepackage{amssymb,latexsym,textcomp}
\usepackage{xcolor}

\usepackage{cite} 

\usepackage{hyperref}
\hypersetup{colorlinks,allcolors=blue}
\usepackage[OT2,T1]{fontenc}

\theoremstyle{plain}
\newtheorem{theorem}{Theorem}[section]
\newtheorem{lemma}[theorem]{Lemma}

\theoremstyle{definition}
\newtheorem{definition}{Definition}[section]
\newtheorem{example}{Example}[section]

\theoremstyle{remark}
\newtheorem{remark}{Remark}[section]

\DeclareMathOperator{\diam}{diam}
\DeclareMathOperator{\Ker}{Ker}
\DeclareMathOperator{\diag}{diag}

\newcommand{\Lin}{Lin}
\newcommand{\Img}{\mathrm{Im}}
\newcommand{\strongperp}{\mathrel{\perp\!\!\!\!\perp}^S}
\newcommand{\strongp}{\mathrel{\perp}^S}

\numberwithin{equation}{section} 

\begin{document}
\title[Diameter of the $BJ$-orthograph in finite dimensional $C^*$-algebras]{Diameter of the $BJ$-orthograph in finite dimensional $C^*$-algebras} 

\author{Srdjan Stefanovi\' c}
\address{University of Belgrade\\ Faculty of Mathematics\\ Student\/ski trg 16-18\\ 11000 Beograd\\ Serbia}

\email{srdjan.stefanovic@matf.bg.ac.rs}

\thanks{The research was supported by the Serbian Ministry of Education, Science and Technological Development through Faculty of Mathematics, University of Belgrade.}

\begin{abstract} We determine the exact diameter of the orthograph related to mutual strong Birkhoff-James orthogonality in arbitrary finite dimensional $C^*$-algebra. Additionally, we will estimate the distance between vertices in an arbitrary $C^{*}$-algebra that can be represented as a direct sum.
\end{abstract}




\maketitle

\section{Introduction and preliminaries}

Orthogonality problems have been studied for many years. In general, normed and Banach spaces are not equipped with an inner product. Therefore, the question arises how to define the relation of orthogonality in spaces where we do not have the inner product but only a norm. There are different ways, such as Robert's orthogonality or Pythagoras orthogonality (see for example \cite{ArambasicRoberts2019} and \cite{Magajna2022}). Probably, the most common type of orthogonality is Birkhoff-James orthogonality, defined by Birkhoff in \cite{Birkhoff} and developed by James (see \cite{James3},\cite{James2},\cite{James1}).

\begin{definition} Let $X$ be a normed space, and let $x$, $y\in X$. We say that $x$ is Birkhoff-James orthogonal to $y$, and denote $x\perp_{BJ} y$, or simply $x\perp y$, if for any $\lambda\in \mathbb C$ there holds
\begin{equation}\nonumber\label{BJord}
\|x+\lambda y\|\ge\|x\|.
\end{equation}

\end{definition}

Birkhoff-James orthogonality has some deficiencies, for instance, it is neither symmetric nor additive. Birkhoff-James orthogonality will be abbreviated to $BJ$-orthogonality hereinafter.

For further details on $BJ$-orthogonality, the reader is referred to a recent survey \cite{ArambasicSurvey} and references therein.

Besides usual $BJ$-orthogonality, there is also a notion of strong $BJ$-orthogonality where the scalar is replaced with an element from the Hilbert $C^*$-module. 

\begin{definition} Let $X$ be a right Hilbert $C^*$-module over a $C^{*}$-algebra $A$ and let $a$, $b\in X$.

a) We say that $a$ is strong $BJ$-orthogonal to $b$, and denote $a\strongp b$ if for any $c\in A$ there holds
\begin{equation}\nonumber
\|a+bc\|\ge\|a\|.
\end{equation}

b) We say that $a$ and $b$ are mutually strong $BJ$-orthogonal to each other, and denote $a\strongperp b$ if $a\strongp b$ and $b\strongp a$.
\end{definition}

\begin{remark}
    The part a) of the previous definition is from \cite{ArambasicAFA2014}, and part b) from \cite{ArambasicBJMA2020}.
\end{remark}

In particular, in a $C^{*}$-algebra A, regarded as a right Hilbert $C^{*}$-module over itself, $a\in A$ is strong Birkhoff–James orthogonal to $b\in A$ if for any $c\in A$ there holds
\begin{equation}\nonumber
\|a+bc\|\ge\|a\|.
\end{equation}
It is easy to show that 
$$a^{*}b=0\Rightarrow a\strongperp b\Rightarrow a\perp_{BJ} b.$$

Recently, in \cite{ArambasicBJMA2020} the orthograph related to $BJ$-orthogonality has been introduced. 

\begin{definition} Let $A$ be a $C^*$-algebra, and let $S$ be the corresponding projective space, i.e.\ $S=(A\setminus\{0\})/a\sim \lambda a$, where $\lambda$ runs through $\mathbb C$.

The orthograph $\Gamma(A)$ related to the mutual strong $BJ$-orthogonality is the graph with $S$ as a set of vertices and with edges consisting of those pairs $(a,b)\in S\times S$ for which $a\strongperp b$.
\end{definition}

Then in \cite{Keckic2023JMAA} it is shown what isolated vertices in $\Gamma(A)$ are, which $C^{*}$-algebras have only one connected component and that diameter of arbitrary $C^{*}$-algebras is at most 4. At the end of the article, the problem of determining the exact diameter for finite dimensional $C^{*}$-algebras is posed. In this paper, we give a complete answer to this problem and we solve some related problems connected with distance between non-isolated vertices in ortograph in arbitrary $C^{*}$-algebra. 

First, in standard books dealing with $C^{*}$-algebras, like \cite{Murphy} and \cite{Pedersen}, the structure of finite dimensional $C^{*}$-algebras is given.
\begin{theorem}
    Let A be a finite dimensional $C^{*}$-algebra. Then A is isomorphic to a direct sum of matrix algebras, i.e., there exist positive integers $n_1, n_2,\dots, n_k$ such that
    $$A\cong M_{n_1}(\mathbb{C})\oplus M_{n_2}(\mathbb{C})\oplus\dots M_{n_k}(\mathbb{C}).$$
\end{theorem}
Recall that for unital $C^{*}$-algebras $A$ and $B$, $A\oplus B$ is $C^{*}$-algebra equipped with a norm 
$$\|a\oplus b\|=\max\{\|a\|,\|b\|\}.$$
We will use a shorter notation $(a,b)$ instead of $a\oplus b$. Recall that from Lemma 5.3.  from \cite{Keckic2023JMAA}, we know that $(a_1,a_2,\dots,a_k)$ is non-isolated vertex in unital $C^{*}$-algebra iff there is some $a_i$ that is non-invertible.

The following theorem is crucial in examining whether two matrices are mutually strong orthogonal or not (see \cite{ArambasicAFA2014}, Proposition 2.8.).
\begin{theorem}\label{karakterizacija}
    Let $a,b\in M_n(\mathbb{C})$. Then $a\strongp b$ if there exists a unit vector $x\in\mathbb{C}^n$ such that 
    $$\|ax\|=\|a\|\ \text{and}\ b^{*}ax=0.$$
\end{theorem}

In paper \cite{ArambasicBJMA2020}, among other things, the diameter of the orthograph for the $C^{*}$-algebra $B(H)$, where $H$ is the Hilbert space, is determined. We will use Theorem 2.9. from \cite{ArambasicBJMA2020} in the matrix framework.

\begin{theorem}\label{B(H)}
    Let $\mathcal{S}$ be the set of all nonzero non-invertible matrices in $M_n(\mathbb{C})$ (up to the scalar multiplication). Then:
    \begin{itemize}
        \item[(1)] $a$ is an isolated vertex of $\Gamma(M_n(\mathbb{C}))$ if and only if $a$ is invertible.
        \item[(2)] If $n=1$, $S=\emptyset$.
        \item[(3)] If $n=2$, then $S$ in disconnected. The connected components of the orthograph $\Gamma(M_2(\mathbb{C}))$ are either isolated vertices or the sets of the form
        $$\mathcal{S}_x=\{a\in M_n(\mathbb{C})\ |\ \Img\, a=\Lin \{x\}\ \text{or}\ \Img\, a=\Lin\{x\}^{\perp}\}.$$
        \item[(4)] If $n=3$, then $S$ is a connected component whose diameter is 4. 
        \item[(5)] If $n\geqslant4$, then $S$ is a connected component whose diameter is 3. 
    \end{itemize}
\end{theorem}

In paper \cite{Keckic2023JMAA}, the following lemma is proved.
\begin{lemma}\label{ortogonalni po koordinatama}
    If $a_1\strongperp a_2$ and $b_1\strongperp b_2$ then $(a_1,b_1)\strongperp(a_2,b_2)$. In particular
        $$(a,\textbf{0})\strongperp(\textbf{0},b)$$
    whenever $a$, $b\neq\textbf{0}$. 
\end{lemma}
Of course, it simply translates to the case with a finite number of coordinates.

The last lemma is not valid in reverse direction, and that is the main problem in direct sum $C^{*}$-algebra. 
\begin{example}
    Namely, notice that if $I$ is an identity matrix and $P$ is rank one projection in $M_2(\mathbb{C})$, then $(I,P)\strongperp(P,I)$, but $I\not\strongperp P$, since $I\strongp P$ and $P\not\strongp I$. 
\end{example}

\section{Distance between vertices in direct sum $C^{*}$-algebras}

We suppose that all $A_1,\dots,A_k$ are unital $C^{*}$-algebras. Some statements can easily be expressed in non-unital case, but the goal of this paper are finite-dimensional algebras that are necessarily unital. In this section we consider the direct sum $C^{*}$-algebra $A=A_1\oplus A_2\oplus\dots\oplus A_k$, where $k\geqslant2$.

\begin{theorem}\label{neinvertibilni na razlicitim koordinatama}
     If there exist two not right invertible elements $a_i\neq\textbf{0}$, $b_j\neq\textbf{0}$ and $(i\neq j)$, then distance between $(a_1,a_2,\dots,a_k)\ \text{and}\ (b_1,b_2,\dots,b_k)\in A$ is at most 3. 
\end{theorem}
\begin{proof}
    Since $a_i$ and $b_j$ are not right invertible, there are $a_i'\neq\textbf{0}$ and $b_j'\neq\textbf{0}$ such that $a_i\strongperp a_i'$ (in $A_i$) and $a_j\strongperp a_j'$ (in $A_j$) (see \cite{Keckic2023JMAA}, Propostion 2.4.). Then, it is true that
    $$(a_1,\dots,a_i,\dots,a_k)\strongperp(\textbf{0},\dots,a_i',\dots,\textbf{0})\strongperp(\textbf{0},\dots,b_j',\dots,\textbf{0})\strongperp(b_1,b_2,\dots,b_k),$$
    by Lemma \ref{ortogonalni po koordinatama} because all the elements are mutually strong orthogonal in terms of coordinates. 
\end{proof}

Since we will be interested in when the diameter of the orthograph will be greater than 3, we can assume that there are no right non-invertible in different places. Therefore, in the following theorems, we will observe the cases in which right non-invertibles are at the same position, and it will be important whether the norm is reached on them or not.

\begin{theorem}\label{najvec1i neinvertibilan}
    Let $(a_1,a_2,\dots,a_k),(b_1,b_2,\dots,b_k)\in A$ be mutual strong $BJ$-ortho\-go\-nal and $a_k$ is the only not right invertible element  $(\text{in}\ (a_1,a_2,\dots,a_k))$. Suppose that $\|a_k\|>\|a_i\|$ for all $1\leqslant i\leqslant k-1$. Then $a_k\strongperp b_k\ (in\ A_k)$  and $\|b_k\|\geqslant\|b_i\|$ for all $1\leqslant i\leqslant k-1$ (especially, $b_k\neq0)$.
\end{theorem}
\begin{proof}
    We know that 
    $$\max\{(\|a_1+b_1c_1\|,\|a_2+b_2c_2\|,\dots,\|a_k+b_kc_k\|)\}\geqslant\|a_k\|,$$
    for all $(c_1,c_2,\dots,c_k)\in A$, so if we put $c_1=c_2=\dots=c_{k-1}=\textbf{0}$, we get $\|a_k+b_kc_k\|\geqslant\|a_k\|$ for all $c_k\in A_k$. So, $a_k\strongp b_k$.

    Next, we know that 
    $$\max\{\|b_1+a_1c_1\|,\|b_2+a_2c_2\|,\dots,\|b_k+a_kc_k\|\}\geqslant\max\{\|b_1\|,\|b_2\|,\dots,\|b_k\|\}.$$
    If we put $c_i=-a_i^{-1}b_i$ ($a_i^{-1}$ is right inverse of $a_i$) for all $1\leqslant i\leqslant k-1$, we get 
    $$\|b_k+a_kc_k\|\geqslant\max\{\|b_1\|,\|b_2\|,\dots,\|b_k\|\}\geqslant\|b_k\|,$$
    so $b_k\strongp a_k$ and if we put $c_k=0$, we get $\|b_k\|\geqslant\|b_i\|$ for all $1\leqslant i\leqslant k-1$.
\end{proof}

\begin{lemma}\label{neinvertibilni najveci dovoljno za ortogonalnost}
    Let $(a_1,a_2,\dots,a_k),(b_1,b_2,\dots,b_k)\in A$. Suppose that $\|a_k\|\geqslant\|a_i\|$, $\|b_k\|\geqslant\|b_i\|$ for all $1\leqslant i\leqslant k-1$ and $a_k\strongperp b_k$. Then 
    $$(a_1,a_2,\dots,a_k)\strongperp(b_1,b_2,\dots,b_k).$$
    \begin{proof}
        For all $c_k\in A_k$ it is true that $\|a_k+b_kc_k\|\geqslant\|a_k\|$ and $\|b_k+a_kc_k\|\geqslant\|b_k\|$. So 
        \begin{equation}
            \begin{split}\nonumber
                &\max\{\|a_1+b_1c_1\|,\|a_2+b_2c_2\|,\dots,\|a_k+b_kc_k\|\}\\
                &\geqslant\|a_k+b_kc_k\|\geqslant\|a_k\|=\max\{\|a_1\|,\|a_2\|,\dots,\|a_k\|\},
            \end{split}
        \end{equation}
        so $(a_1,a_2,\dots,a_k)\strongp(b_1,b_2,\dots,b_k)$. The second direction is proved analogously.
    \end{proof}
\end{lemma}

\begin{theorem}\label{neinvertibilan nenajveci}
    Let $(a_1,a_2,\dots,a_k),(b_1,b_2,\dots,b_k)\in A$ such that $a_k$ and $b_k$ are only not right invertible and there exist $a_i(i\neq k)$ and $b_j(j\neq k)$ such that $\|a_i\|\geqslant\|a_k\|$ and $\|b_j\|\geqslant\|b_k\|$. Then the distance between $(a_1,a_2,\dots,a_k)$ and $(b_1,b_2,\dots,b_k)$ is at most 2. 
\end{theorem}
\begin{proof}
    Notice that $(a_1,a_2,\dots,a_k)\strongperp(\textbf{0},\textbf{0},\dots,\textbf{1})\strongperp(b_1,b_2,\dots,b_k)$. It is enough to prove the first one. For all $(x_1,x_2,\dots,x_k)\in A$ it is true that 
    \begin{equation}
        \begin{split}\nonumber
            \max\{\|a_1+\textbf{0}x_1\|,\|a_2+\textbf{0}x_2\|,\dots,\|a_k+\textbf{1}x_k\|\}&\geqslant\max\{\|a_1\|,\|a_2\|,\dots,\|a_{k-1}\|\}\\
            &=\max\{\|a_1\|,\dots,\|a_{k-1}\|,\|a_{k}\|\},
        \end{split}
    \end{equation}
    where the last equality is valid because $\|a_k\|\leqslant\|a_i\|$.

    Further, in every unital $C^{*}$-algebra it is true that $\textbf{1}\strongp a_k$ for every not right invertible $a_k$. Indeed, there is a pure state $\rho$ such that $\rho(a_ka_k^{*})=0$ because $a_k$ is not right invertible. In the end, $\rho(\textbf{1}\textbf{1}^{*})=1=\|\textbf{1}\|^2$, because $\rho$ is state. So, by Lemma 2.3. from \cite{Keckic2023JMAA}, we get $\textbf{1}\strongp a_k$. From there, for any $x_k\in A_k$, $\|\textbf{1}+a_kx_k\|\geqslant\|\textbf{1}\|$, and finally for every $(x_1,x_2,\dots,x_k)\in A$ it is true that
    $$\max\{\|\textbf{0}+a_1x_1\|,\|\textbf{0}+a_2x_2\|,\dots,\|\textbf{1}+a_kx_k\|\}\geqslant\|\textbf{1}\|=\max\{\|\textbf{0}\|,\|\textbf{0}\|,\dots,\|\textbf{1}\|\}.$$
\end{proof}

\begin{theorem}\label{kombinacija neinvertibilnih}
    Let $(a_1,a_2,\dots,a_k),(b_1,b_2,\dots,b_k)\in A$ such that $a_k$ and $b_k$ are only not right invertible. Further, let there be $a_i(i\neq k)$ such that $\|a_i\|\geqslant\|a_k\|$ and let $\|b_k\|>\|b_j\|$ for every $j\in[1,k-1]$. Then distance between $(a_1,a_2,\dots,a_k)$ and $(b_1,b_2,\dots,b_k)$ is at most 3. 
\end{theorem}
\begin{proof}
    Because $b_k$ is not right invertible, there is non-zero not right invertible $b_k'$ such that $b_k\strongperp b_k'$. But then it is true
    $$(a_1,a_2,\dots,a_k)\strongperp(\textbf{0},\textbf{0},\dots,\textbf{1})\strongperp(\|b_k'\|\textbf{1},\|b_k'\|\textbf{1},\dots,b_k')\strongperp(b_1,b_2,\dots,b_k).$$
    First mutual strong $BJ$-orthogonality is proved in Theorem \ref{neinvertibilan nenajveci}, as well as the first direction of the second mutual strong $BJ$-orthogonality. Reverse direction is trivial because for all $(x_1,x_2,\dots,x_k)\in A$
    \begin{equation}
        \begin{split}\nonumber
            &\max\{\|\|b_k'\|\textbf{1}+\textbf{0}x_1\|,\|\|b_k'\|\textbf{1}+\textbf{0}x_2\|,\dots,\|b_k'+\textbf{1}x_k\|\}\\
            &\geqslant\|b_k'\|=\max\{\|\|b_k'\|\textbf{1}\|,\|\|b_k'\|\textbf{1}\|,\dots,\|b_k'\|\}.
        \end{split}
    \end{equation}
 Third mutual strong $BJ$-orthogonality is true by Lemma \ref{neinvertibilni najveci dovoljno za ortogonalnost}, which concludes the proof.
\end{proof}
\begin{remark}\label{nezgodno C}
    Therefore, the only case in which the distance can be greater than 3 (then it is 4, because it is always smaller than 5) is when both right non-invertibles are such that the norm is reached on them.
\end{remark}

\section{Diameter of finite dimensional $C^*$-algebras}

According to the remark \ref{nezgodno C}, the problem arises in the case in which the Theorem \ref{najvec1i neinvertibilan} should be applied. In the case where $A_k=\mathbb{C}$, there is no non-invertible such that its norm is greater than the other elements ($0$ is the only one). Also, the problem in the same case is to determine the diameter of the algebra $A_k$. We distinguish the cases when it is $\mathbb{C}, M_2(\mathbb{C}), M_3(\mathbb{C})$ or $M_n(\mathbb{C}), n\geqslant4$. Let us recall that in $M_n(\mathbb{C})$ the fact that an element is not right invertible is equivalent to the fact that it is non-invertible.

\subsection{Case of $\mathbb{C}^{k}, k\in\mathbb{N}$}
\subsubsection{Case of $\mathbb{C}\oplus\mathbb{C}$}
There are only two nonisolated vertices in this $C^{*}$-algebra, $(0,1)$ and $(1,0)$. They are mutual strong orthogonal, so, there is only one connected component and its diameter is 1.

\subsubsection{Case of $\mathbb{C}^{k}, k\geqslant3$}
We will show that $\diam(\mathbb{C}^k)=3$ for all $k\geqslant3$. If $a_k$ and $b_k$ are non-invertible (so $a_k=b_k=0$), by Theorem \ref{neinvertibilan nenajveci} their distance is at most 2. And if they are on different positions, their distance is at most 3 by Theorem \ref{neinvertibilni na razlicitim koordinatama}.

It remains to find the vertices whose distance is 3. 
\begin{example}
    The distance between $(0,1,2,1,\dots,1)$ and $(2,0,1,1,\dots,1)$ is 3. 

    If $(0,1,2,1,\dots,1)\strongperp(a_1,a_2,a_3,\dots,a_n)$, then for all $\lambda_1,\lambda_2,\dots,\lambda_n\in\mathbb{C}$ 
    $$\|(0+a_1\lambda_1,1+a_2\lambda_2,2+a_3\lambda_3,\dots,1+\lambda_na_n)\|\geqslant2,$$
    and if $a_3\neq0$, we just put $\lambda_3=-\frac{2}{a_3}$ and $\lambda_i=0$ for $i\neq3$ and get a contradiction. So $a_3=0$. Next, for all $\lambda_1,\lambda_2,\dots,\lambda_n\in\mathbb{C}$ it must be true 
    $$\|(a_1+0\lambda_1,a_2+1\lambda_2,0+2\lambda_3,\dots,a_n+1\lambda_n)\|\geqslant\max\{\|a_1\|,\|a_2\|,\dots,\|a_n\|\}.$$
    Now put $\lambda_k=-a_k$ for $k\neq3$ and $\lambda_3=0$ and get $\|a_1\|\geqslant\|a_k\|$ for all $k\in[1,n]$. 

    In the same way, if $(2,0,1,1,\dots,1)\strongperp(b_1,b_2,b_3,\dots,b_n)$ we get $b_1=0$ and $\|b_2\|\geqslant\|b_k\|$ for all $k\in[1,n]$. 

    So it is not true that $(0,1,2,1,\dots,1)\strongperp(2,0,1,1,\dots,1)$.

    If there is some $(a_1,a_2,a_3,\dots,a_n)$ mutual strong orthogonal to both vertices, than $a_1=0$ and its norm is greater or equal to all others, so, $a_n=0$ for all $n\in\mathbb{N}$.

    Then the distance is greater than 2, which implies that it is 3.
\end{example}

So we proved 
\begin{lemma}\label{Lema 1}
    $\diam(\mathbb{C}^2)=1$ and $\diam(\mathbb{C}^k)=3$ for $k\geqslant3$. 
\end{lemma}

\begin{remark}
    It is interesting that diameter in $\mathbb{C}^3$ (with a $\max$ norm) is smaller than the diameter in orthograph related to standard Birkhoff-James orthogonality (which is weaker) considered in \cite{Kuzma2021JMAA}. In that case, the diameter is 4 (see Propostion 5.15). This is because in the case of strong $BJ$-orthogonality there are more isolated vertices, so despite the stronger relation, the number of vertices in connected component is smaller. 
\end{remark}

\begin{theorem}\label{dijametar veci od 2}
    Let $A=M_{n_1}(\mathbb{C})\oplus M_{n_2}(\mathbb{C})\oplus\dots M_{n_k}(\mathbb{C}),\ k\geqslant2$ such that $A\not\cong\mathbb{C}^n$ for an arbitrary $n\in\mathbb{N}$. Then $\diam A\geqslant3$.
\end{theorem}
\begin{proof}
    Let us assume, say $M_{n_k}(\mathbb{C})$ such that $n_k\geqslant2$. By Theorem \ref{B(H)}, we know that diameter of $M_{n_k}(\mathbb{C})$ is bigger than 2, or has two disconnected components (in the case of $M_2(\mathbb{C})$). Be that as it may, we can choose two non-invertible, non-zero elements $a_k,b_k\in M_{n_k}(\mathbb{C})$ such that their distance is bigger than 2. If we choose all other $a_i,b_i\ (i\in[1,k-1])$ to be invertible and with a smaller norm than $\|a_k\|$ and $\|b_k\|$, respectively, by Theorem \ref{najvec1i neinvertibilan} we immediately conclude that $\diam A>2$. Otherwise, $(a_1,a_2,\dots,a_k)$ and $(b_1,b_2,\dots,b_k)$ are mutual strong $BJ$-orthogonal or there is $(c_1,c_2,\dots,c_k)$ mutual strong $BJ$-orthogonal to both. But again by Theorem \ref{najvec1i neinvertibilan}, this is impossible, because distance between $a_k$ and $b_k$ is bigger than 2.  
\end{proof}

For the sake of clarity, we will first examine cases with two summands, and then move on the cases with three or more.

\subsection{Cases with two summands}
\subsubsection{Case of $M_n(\mathbb{C})\oplus M_k(\mathbb{C}),n,k\geqslant2$}
We will prove that $\diam(M_n(\mathbb{C})\oplus M_k(\mathbb{C}))$ is equal to 3. We know by Theorem \ref{dijametar veci od 2} that diameter is at least 3, and by theorems \ref{neinvertibilni na razlicitim koordinatama}, \ref{najvec1i neinvertibilan}, \ref{neinvertibilan nenajveci}, \ref{kombinacija neinvertibilnih}, we must examine only the case where non-invertible elements are at the same place and have strictly bigger norm than invertible elements.
\begin{lemma}
    Let $(a,b),(c,d)\in M_n(\mathbb{C})\oplus M_k(\mathbb{C})$ such that $a,b$ are non-invertible and $c,d$ are invertible matrices. Further, suppose that $\|a\|>\|b\|$ and $\|c\|>\|d\|$. Then the distance between them is at most 3.
\end{lemma}
\begin{proof}
        

    Because $a$ and $c$ are non-invertible, there are non-zero, non-invertible $a_1$ and $c_1$ such that $a\strongperp a_1$ and $c\strongperp c_1$. Then it is true
    $$(a,b)\strongperp(a_1,\diag(\|a_1\|,0,\dots,0))\strongperp(c_1,\diag(0,\dots,0,\|c_1\|))\strongperp(c,d),$$
    where $\diag(\|a_1\|,0,\dots,0)$ denotes $k\times k$ diagonal matrix which only non-zero entry is in the first row and the first column and it is $\|a_1\|$.

    First and third mutual strong $BJ$-orthogonality is valid by Lemma \ref{neinvertibilni najveci dovoljno za ortogonalnost}.


    Finally, let us note that $(\diag(\|a_1\|,0,\dots,0))^{*}\diag(0,\dots,0,\|c_1\|)=\textbf{0}$, so they are mutual strong $BJ$-orthogonal. Again, by Lemma \ref{neinvertibilni najveci dovoljno za ortogonalnost}, second $BJ$-orthogonality is valid.
\end{proof}

Thus, we proved 
\begin{lemma}\label{Lema 2}
    $\diam(M_n(\mathbb{C})\oplus M_k(\mathbb{C}))=3$, for all $n,k\geqslant2$.
\end{lemma}

\subsubsection{Case of $\mathbb{C}\oplus M_2(\mathbb{C})$}
We will prove that $\diam(\mathbb{C}\oplus M_2(\mathbb{C}))=4$. Because in every $C^{*}$-algebra diameter is at most 4, it is enough to find two vertices whose distance is exactly 4. 
\begin{example}
Distance between $\left(1,\begin{bmatrix}
    2&0\\
    0&0
\end{bmatrix}\right)$ and $\left(1,\begin{bmatrix}
    1&1\\
    1&1
\end{bmatrix}\right)$ is 4.  

Norm of both matrices is 2. By Theorem \ref{najvec1i neinvertibilan}, we know that these two are not mutual strong orthogonal. Moreover, if there exist $(a, b)\in\mathbb{C}\oplus M_2(\mathbb{C})$ that is mutual strong orthogonal to both, then again by Theorem \ref{najvec1i neinvertibilan}, $b\strongperp\begin{bmatrix}
    2&0\\
    0&0
\end{bmatrix}$ and $b\strongperp\begin{bmatrix}
    1&1\\
    1&1
\end{bmatrix}$ and $\|b\|\geqslant\|a\|$. This is not possible, because by Theorem \ref{B(H)}(3) $\begin{bmatrix}
    2&0\\
    0&0
\end{bmatrix}b=\textbf{0}$ and $\begin{bmatrix}
    1&1\\
    1&1
\end{bmatrix}b=\textbf{0}$. If $b=\begin{bmatrix}
    b_{11}&b_{12}\\
    b_{21}&b_{22}
\end{bmatrix}$, then $2b_{11}=0,\ 2b_{12}=0,\ b_{11}+b_{21}=0$ and $b_{12}+b_{22}=0$, so $b=\textbf{0}$, by then $a=0$ which is not possible.

In the end, if we suppose that there exist two non-zero $(a,b),(c,d)\in\mathbb{C}\oplus M_2(\mathbb{C})$ such that
$$\left(1,\begin{bmatrix}
    2&0\\
    0&0
\end{bmatrix}\right)\strongperp(a,b)\strongperp(c,d)\strongperp \left(1,\begin{bmatrix}
    1&1\\
    1&1
\end{bmatrix}\right),$$
we know (by Theorem \ref{najvec1i neinvertibilan}) that $b\strongperp\begin{bmatrix}
    2&0\\
    0&0
\end{bmatrix}$ and $d\strongperp \begin{bmatrix}
    1&1\\
    1&1
\end{bmatrix}$, but also $\|b\|\geqslant\|a\|$ and $\|d\|\geqslant\|c\|$. There are three cases:
\begin{itemize}
    \item[1.] $a=c=0$: Then $b\strongperp d$, but this is not possible, because the images of matrices are not the same or orthogonal in the usual sense (and both are different from $\textbf{0}$).
    \item[2.] $a$ and $c$ are both invertible: Then again it must be $b\strongperp d$, which is again not possible. Indeed, it is true that for all $(x,y)\in\mathbb{C}\oplus M_2(\mathbb{C})$
    $$\max\{(|a+cx|,\|b+dy\|)\}\geqslant\max\{|a|,\|b\|\}=\|b\|,$$
    and if we put $x=-c^{-1}a$, we get $b\strongp d$. In the same way, we prove the other direction.
    \item[3.] $a$ is invertible and $c=0$ (the other direction is done analogously). Then for all $(x,y)\in\mathbb{C}\oplus M_2(\mathbb{C})$
    $$\max\{|0+ax|,\|d+by\|\}\geqslant\max\{|0|,\|d\|\}=\|d\|,$$
    so if $x=0$, it must be $d\strongp b$. But this is not possible either (strong $BJ$-orthogonality in one direction). Indeed, 
we know that $b_{11}=b_{12}=0$ (for the same reason as in the first part of the case when we considered that the distance is 2) and $\Img\left(\begin{bmatrix}
    1&1\\
   1&1
\end{bmatrix}\right)=\left\{\lambda\begin{bmatrix}
    1\\
    1
\end{bmatrix}\ |\ \lambda\in\mathbb{C} \right\}$, so $\Img\, d=\left\{\lambda\begin{bmatrix}
   1\\
    -1
\end{bmatrix}\ |\ \lambda\in\mathbb{C} \right\}$ (because $d\strongperp \begin{bmatrix}
    1&1\\
    1&1
\end{bmatrix}$). Because $d\strongp b$, by Theorem \ref{karakterizacija} there exists a unit vector $y\in\mathbb{C}^2$ such that $\|dy\|=\|d\|$ and $b^{*}dy=\textbf{0}$. So, knowing the image of the matrix $d$, $dy=\lambda\begin{bmatrix}
   1\\
    -1
\end{bmatrix}$ for some $\lambda\neq0$, and then $b^{*}\begin{bmatrix}
   1\\
    -1
\end{bmatrix}=\begin{bmatrix}
    -\overline{b_{21}}\\
    -\overline{b_{22}}
\end{bmatrix}$, so $b=\textbf{0}$ which is impossible.  
\end{itemize}
\end{example}
So, it is true:
\begin{lemma}\label{Lema 3}
    $\diam(\mathbb{C}\oplus M_2(\mathbb{C}))=4$.
\end{lemma}

\subsubsection{Case of $\mathbb{C}\oplus M_3(\mathbb{C})$} We will show that diameter is 3 in this case. As in the previous cases, we only have to consider the situation when the non-invertibles are at the same position and their norms are strictly greater than the norm of invertible ones. This is possible only if they are non-invertible in $M_3(\mathbb{C})$.

\begin{lemma}
    Let $(a,b),(c,d)\in \mathbb{C}\oplus M_3(\mathbb{C})$ such that $a,b$ are invertible (so different from $0$) and $c,d$ are non-invertible matrices. Further, suppose that $\|a\|<\|b\|$ and $\|c\|<\|d\|$. Then the distance between them is at most 3.
\end{lemma}
\begin{proof}
    The idea is to construct $e,f\in M_3(\mathbb{C})$ such that $b\strongperp e$ and $d\strongperp f$, but with the additional property that $e\strongp f$ (note that mutual strong $BJ$-orthogonality is not possible in the general case). Then it is true that
    $$(a,b)\strongperp(0,e)\strongperp(\|f\|,f)\strongperp(c,d).$$
    Namely, first and third mutual strong $BJ$-orthogonality is valid by Lemma \ref{neinvertibilni najveci dovoljno za ortogonalnost}, as well as the second in the left-to-right direction. Also, for all $(x,y)\in\mathbb{C}\oplus M_3(\mathbb{C})$ it is true that 
    $$\max\{|\|f\|+0x|,\|f+ey\|\}\geqslant\|f\|=\max\{|\|f\||,\|f\|\},$$
    which proves the missing direction. So, what remains is to construct $e$ and $f$.

    There exist unit vectors $x,y\in\mathbb{C}^3$ such that $\|bx\|=\|b\|$ and $\|dy\|=\|d\|$. Also, because $b$ and $d$ are non-invertible, there exists unit vectors $v_b\in \Ker b^{*}$ and $v_d\in \Ker d^{*}$. There exist unit vector $w\in\Lin\{bx,v_d\}^{\perp}$ (dimension is 3). We define $e$ with $ev_b=v_b,ew=w$ and $e|_{\Lin\{v_b,w\}^{\perp}}=\textbf{0}$ and $f$ with $fv_d=v_d,f|_{\{v_d\}^{\perp}}=\textbf{0}$. 

    Let us notice that $\langle bx,ev_b\rangle=\langle bx,v_b\rangle=\langle x,b^{*}v_b\rangle=\langle x,0\rangle=0$ and $\langle bx,ew\rangle=\langle bx,w\rangle=0$ because $bx\perp w$. So, $bx\perp \Img\, e$ from where we conclude $e^{*}bx=\textbf{0}$. As $\|bx\|=\|b\|$, by Theorem \ref{karakterizacija} we get $b\strongp e$. Further, we know that it holds $\|ev_b\|=\|v_b\|=1=\|e\|$ and $b^{*}ev_b=b^{*}v_b=\textbf{0}$, because $v_b\in \Ker b^{*}$. Again, by Theorem \ref{karakterizacija} we get $b\strongperp e$. In a similar way, we get $dy\perp \Img\,f$, so $f^{*}dy=\textbf{0}$. As $\|dy\|=\|d\|$, it holds $d\strongp f$. Also, $\|fv_d\|=\|v_d\|=1=\|f\|$ and $d^{*}fv_d=d^{*}v_d=\textbf{0}$, so $d\strongperp f$. In the end, $\|ew\|=\|w\|=1=\|e\|$ and $f^{*}ew=f^{*}w=\textbf{0}$ because $w\perp v_d$, so we proved $e\strongp f$, which completes the proof.
\end{proof}

Thus we proved 
\begin{lemma}\label{Lema 4}
    $\diam(\mathbb{C}\oplus M_3(\mathbb{C}))=3$.
\end{lemma}

\subsubsection{Case of $\mathbb{C}\oplus M_n(\mathbb{C}), n\geqslant4$}
By Theorem \ref{dijametar veci od 2} is at least 3. By Theorem \ref{najvec1i neinvertibilan}, we need to examine only the case where non-invertibles are the biggest in respect to norm and at the same position. The only possible situation is that they are in $M_n(\mathbb{C})$ ($0$ is only non-invertible in $\mathbb{C}$ and has norm $0$). Finally, by Theorem \ref{B(H)}, the diameter of $M_n(\mathbb{C})$ is 3, so diameter of $\mathbb{C}\oplus M_n(\mathbb{C})$ is at most 3 (we can connect the matrices by the path of length at most 3, and write $0$ in the remaining places in part of $\mathbb{C}$).

So, we proved 
\begin{lemma}\label{Lema 5}
    $\diam(\mathbb{C}\oplus M_n(\mathbb{C}))=3$ for all $n\geqslant4$.
\end{lemma}

\subsection{Case with three or more summands}
We will prove that the diameter of these finite dimensional $C^{*}$-algebras is equal to 3 in case $A\not\cong \mathbb{C}^{n}$ for any $n\in\mathbb{N}$, which is the only case left to examine. So, we can assume that some $n_i$ is greater than 1 and that non-invertible elements $a_i,b_i\in M_{n_i}$ have strictly greater norm with respect to other elements. 

\begin{lemma}
    Let $k\geqslant3,n_i\geqslant2$ and $$(a_1,\dots,a_i,\dots,a_k),(b_1,\dots,b_i,\dots,b_k)\in A=M_{n_1}(\mathbb{C})\oplus\dots M_{n_i}(\mathbb{C})\oplus\dots M_{n_k}(\mathbb{C}),$$
    such that $a_i,b_i$ are the only non-invertible elements and $\|a_i\|>\|a_j\|,\|b_i\|>\|b_j\|$ for all $j\neq i$. Than the distance between them is at most 3.
\end{lemma}
\begin{proof}
    Because $a_i,b_i$ are non-invertible, there are non-zero non-invertible $a_i'$ and $b_i'$ such that $a_i\strongperp a_i'$ and $b_i\strongperp b_i'$. We will prove that
    \begin{equation}
        \begin{split}\nonumber
            (a_1,\dots,a_i,\dots,a_k)&\strongperp(\|a_i'\|\textbf{1},\dots,a_i',\dots,\textbf{0})\\
            &\strongperp(\textbf{0},\dots,b_i'\,\dots,\|b_i'\|\textbf{1})\strongperp(b_1,\dots,b_i,\dots,b_k),
        \end{split}
    \end{equation}
    where all other elements are $\textbf{0}$.

    Namely, first and third mutual strong $BJ$-orthogonality is valid by Lemma \ref{neinvertibilni najveci dovoljno za ortogonalnost}. It remains to prove the second orthogonality. For all $(x_1,\dots,x_i,\dots,x_k)\in A$, it is true that 
    \begin{equation}
        \begin{split}\nonumber
            &\max\{\|\|a_i'\|\textbf{1}+\textbf{0}x_1\|,\dots,\|a_i'+b_i'x_i\|,\dots,\|\textbf{0}+\|b_i'\|\textbf{1}x_k\|\}\\
            &\geqslant\|\|a_i'\|\textbf{1}\|=\max\{\|\|a_i'\|\textbf{1}\|,\dots,\|a_i'\|,\dots,\|\textbf{0}\|\}, 
        \end{split}
    \end{equation}
    which proves that $(\|a_i'\|\textbf{1},\dots,a_i',\dots,\textbf{0})\strongp(\textbf{0},\dots,b_i'\,\dots,\|b_i'\|\textbf{1})$. The second direction is proved analogously.
\end{proof}

So we proved
\begin{lemma}\label{Lema 6}
    Let $A=M_{n_1}(\mathbb{C})\oplus\dots M_{n_i}(\mathbb{C})\oplus\dots M_{n_k}(\mathbb{C}), k\geqslant3$ and $A\not\cong\mathbb{C}^n$ for an arbitrary $n\in\mathbb{N}$. Then $\diam A=3$. 
\end{lemma}

We will conclude this chapter by summarizing the results we have obtained in it, expressing them in one theorem.

\begin{theorem}
 Let $A=M_{n_1}(\mathbb{C})\oplus M_{n_2}(\mathbb{C})\oplus\dots M_{n_k}(\mathbb{C}),k\geqslant2$. Then:
 \begin{itemize}
     \item[(1)] $\diam(\mathbb{C}\oplus\mathbb{C})=1$;
     \item[(2)] $\diam(\mathbb{C}\oplus M_2(\mathbb{C}))=4$;
     \item[(3)] $\diam(A)=3$ in all other cases. 
 \end{itemize}
\end{theorem}
\begin{proof}
The proof is contained in Lemmata \ref{Lema 1}, \ref{Lema 2}, \ref{Lema 3}, \ref{Lema 4}, \ref{Lema 5}, \ref{Lema 6}.        
\end{proof}

\section{Future study}\label{future}
A natural next step in the research will be to find the diameter of $AF$-algebras or $UHF$-algebras (or some classes of them, for example compact operators on Hilbert space), so to see how direct limits and tensor product affect mutual strong $BJ$-orthogonality. There are some recent results in the case of tensor product (see \cite{Mohit}).

\bibliographystyle{abbrv}
\bibliography{Orthograph}

\begin{thebibliography}{10}

\bibitem{ArambasicSurvey}
L.~Aramba{\v{s}}i{\'{c}}, A.~Guterman, B.~Kuzma, and S.~Zhilina.
\newblock {B}irkhoff--{J}ames orthogonality: Characterizations, preservers, and orthogonality graphs.
\newblock In R.~M. Aron, M.~S. Moslehian, I.~M. Spitkovsky, and H.~J. Woerdeman, editors, {\em Operator and Norm Inequalities and Related Topics}, pages 255--302. Springer International Publishing, Cham, 2022.

\bibitem{ArambasicBJMA2020}
L.~Aramba\v{s}i\'c, A.~Guterman, B.~Kuzma, R.~Raji\'c, and S.~Zhilina.
\newblock Orthograph related to mutual strong {B}irkhoff-{J}ames orthogonality in ${C}^*$-algebras.
\newblock {\em Banach J Math. Ann.}, 14(4):1751--1772, 2020.

\bibitem{Kuzma2021JMAA}
L.~Aramba\v{s}i\'c, A.~Guterman, B.~Kuzma, R.~Raji\'c, and S.~Zhilina.
\newblock Symmetrized {B}irkhoff--{J}ames orthogonality in arbitrary normed spaces.
\newblock {\em {J}. {M}ath. {A}nal. {A}ppl.}, 502(1):125203, 2023.

\bibitem{ArambasicAFA2014}
L.~Aramba\v{s}i\'c and R.~Raji\'c.
\newblock A strong version of the {B}irkhoff--{J}ames orthogonality in {H}ilbert ${C}^*$-modules.
\newblock {\em Ann. Funct. Anal.}, 5(1):109--120, 2014.

\bibitem{ArambasicRoberts2019}
L.~Aramba\v{s}i\'c and R.~Raji\'c.
\newblock Roberts orthogonality for complex matrices.
\newblock {\em Acta Math. Hungar.}, 157:220--228, 2019.

\bibitem{Birkhoff}
G.~Birkhoff.
\newblock Orthogonality in linear metric spaces.
\newblock {\em Duke Math. J.}, 1(2):169--172, 1935.

\bibitem{James3}
R.~C. James.
\newblock Orthogonality in normed linear spaces.
\newblock {\em Duke Math. J.}, 12(2):291--302, 1945.

\bibitem{James2}
R.~C. James.
\newblock Inner products in normed linear spaces.
\newblock {\em Bull. Amer. Math. Soc.}, 53:559--566, 1947.

\bibitem{James1}
R.~C. James.
\newblock Orthogonality and linear functionals in normed linear spaces.
\newblock {\em Trans. Amer. Math. Soc.}, 61(2):265--292, 1947.

\bibitem{Keckic2023JMAA}
D.~J. Ke\v{c}ki\'c and S.~Stefanovi\'c.
\newblock Isolated vertices and diameter of the {$BJ$}-orthograph in {$C^{*}$}-algebras.
\newblock {\em {J}. {M}ath. {A}nal. {A}ppl.}, 528(1):127476, 2023.

\bibitem{Magajna2022}
B.~Magajna.
\newblock Orthonormal pairs of operators.
\newblock {\em Linear Algebra Appl.}, 648:233--253, 2022.

\bibitem{Mohit}
Mohit and R.~Jain.
\newblock {B}irkhoff--{J}ames orthogonality in certain tensor products of {B}anach spaces.
\newblock {\em Oper. Matrices}, 17(1):235--244, 2023.

\bibitem{Murphy}
G.~J. Murphy.
\newblock {\em ${C}^*$-Algebras and Operator Theory}.
\newblock Academic Press, San Diego, 1990.

\bibitem{Pedersen}
G.~K. Pedersen.
\newblock {\em ${C}^*$-algebras and their automorphism groups}.
\newblock Pure and Applied Mathematics. Elsevier, Academic Press, London, San Diego, etc., 2018.

\end{thebibliography}

\end{document}